\documentclass[journal,twoside,web]{IEEEtran}
\IEEEoverridecommandlockouts 
\usepackage{graphicx}
\usepackage{balance}
\usepackage{booktabs}
\usepackage{color}
\usepackage{subfigure}
\usepackage{blkarray, bigstrut}
\usepackage{algorithm,algpseudocode}
\usepackage{amsmath, amssymb, amsthm, mathtools, accents}
\usepackage{setspace}
\usepackage{kotex}
\usepackage{multicol}
\usepackage{hyperref}

\usepackage{soul,xcolor}
\usepackage{amsmath,amssymb,amsfonts}
\usepackage{algorithm,algpseudocode}
\usepackage{graphicx}
\usepackage{textcomp}
\usepackage{bbm}
\usepackage{subfigure}
\usepackage{cite}

\usepackage{amsmath, amssymb, mathtools, accents}
\usepackage{algorithm,algpseudocode}
\usepackage{graphicx}          
\usepackage{subfigure}
\usepackage{mathtools}
\usepackage{blkarray, bigstrut}
\usepackage{booktabs}
\usepackage{bbm}

\usepackage[switch,pagewise]{lineno}
\newcommand*\patchAmsMathEnvironmentForLineno[1]{%
\expandafter\let\csname old#1\expandafter\endcsname\csname #1\endcsname
\expandafter\let\csname oldend#1\expandafter\endcsname\csname end#1\endcsname
\renewenvironment{#1}%
{\linenomath\csname old#1\endcsname}%
{\csname oldend#1\endcsname\endlinenomath}}%
\newcommand*\patchBothAmsMathEnvironmentsForLineno[1]{%
\patchAmsMathEnvironmentForLineno{#1}%
\patchAmsMathEnvironmentForLineno{#1*}}%
\AtBeginDocument{%
\patchBothAmsMathEnvironmentsForLineno{equation}%
\patchBothAmsMathEnvironmentsForLineno{align}%
\patchBothAmsMathEnvironmentsForLineno{flalign}%
\patchBothAmsMathEnvironmentsForLineno{alignat}%
\patchBothAmsMathEnvironmentsForLineno{gather}%
\patchBothAmsMathEnvironmentsForLineno{multline}%
}

\usepackage{tikz}
	\usetikzlibrary{shapes,arrows}
	\tikzstyle{frame} = [draw, -latex]
	\tikzstyle{line} = [draw]
	\tikzstyle{line2} = [draw, dashdotted]
	\tikzstyle{line3} = [draw, dashed]
	\tikzstyle{line3UD} = [draw, dashed]
	\tikzstyle{place} = [circle, draw=black, fill=white, thick, inner sep=2pt, minimum size=1mm]
	\tikzstyle{place2} = [circle, draw=black, fill=black, thick, inner sep=2pt, minimum size=1mm]
	\tikzstyle{placeRed} = [circle, draw=red, fill=red, thick, inner sep=2pt, minimum size=1mm]
	\tikzstyle{vertex} = [circle, draw=black, fill=black, thick, inner sep=2pt, minimum size=1mm]
\usepackage{tkz-euclide}

\usetikzlibrary{shapes,arrows} 

\tikzstyle{decision} = [diamond, draw, fill=blue!20,
    text width=4.5em, text badly centered, node distance=3cm, inner sep=0pt]
\tikzstyle{block1} = [rectangle, draw, text width=8em, text centered, minimum height=4em]
\tikzstyle{block2} = [rectangle, draw, text width=3em, text centered, minimum height=4em]
\tikzstyle{block3} = [rectangle, draw, text width=11em, text centered, minimum height=12em, dashed,black]
\tikzstyle{block4} = [rectangle, draw, text width=11em, text centered, minimum height=18em, dashed,black]
\tikzstyle{block5} = [rectangle, draw, text width=11em, text centered, minimum height=32em, dashed,black]
\tikzstyle{block6} = [rectangle, draw, text width=11em, text centered, minimum height=18.5em, dashed,black]
\tikzstyle{block7} = [rectangle, draw, text width=11em, text centered, minimum height=11.8em, dashed,black]
\tikzstyle{line01} = [draw, -latex']
\tikzstyle{line02} = [draw, latex'-latex']

\makeatletter
\def\BState{\State\hskip-\ALG@thistlm}
\makeatother
\algdef{SE}[DOWHILE]{Do}{doWhile}{\algorithmicdo}[1]{\algorithmicwhile\ #1}%

\algnewcommand\algorithmicswitch{\textit{switch}}
\algnewcommand\algorithmiccase{\textit{case}}
\algnewcommand\algorithmicassert{\texttt{assert}}
\algnewcommand\Assert[1]{\State \algorithmicassert(#1)}%
\algdef{SE}[SWITCH]{Switch}{EndSwitch}[1]{\algorithmicswitch\ #1\ }{\algorithmicend\ \algorithmicswitch}%
\algdef{SE}[CASE]{Case}{EndCase}[1]{\algorithmiccase\ #1}{\algorithmicend\ \algorithmiccase}%
\algtext*{EndSwitch}%
\algtext*{EndCase}%

\usepackage{amsmath,bm}

\newtheorem{corollary}{\bfseries Corollary}
\newtheorem{definition}{\bfseries Definition }
\newtheorem{lemma}{\bfseries Lemma}

\newtheorem{remark}{\bfseries Remark}
\newtheorem{theorem}{\bfseries Theorem}
\newtheorem{proposition}{\bfseries Proposition}
\newtheorem{condition}{\bfseries Condition}
\newtheorem{assumption}{\bfseries Assumption}

\title{Controllable Subspaces in Structured Networks of Hierarchical Directed Acyclic Graphs: Controllability of Individual Nodes}
\author{Nam-Jin Park$^{1}$, Yeong-Ung Kim$^{1}$, Koog-Hwan Oh$^{2}$, and Hyo-Sung Ahn$^{1}$
\thanks{${}^{1}$School of Mechanical Engineering, Gwangju Institute of Science and Technology (GIST), Gwangju, Korea. E-mails: {\tt\small namjinpark@gist. ac.kr; yeongungkim@gm.gist.ac.kr; hyosung@gist.ac.kr; }}
\thanks{${}^{2}$Smart Electrics Research Center, Korea Electronics Technology Institute (KETI), Gwangju, Korea. E-mails: {\tt\small ohkhwan@keti.re.kr}}
}

\begin{document}
\maketitle

\begin{abstract} 
This paper introduces the concept of the Fixed Strongly Structurally Controllable Subspace (FSSCS) within the context of structured networks, enabling a comprehensive characterization of controllable subspaces. 
\textcolor{black}{From a graph-theoretical perspective, we define Fixed Strongly Structurally Controllable (FSSC) nodes, 
which are nodes (or states) that remain controllable for all network parameters within a structured network.}
Furthermore, we establish the necessary and sufficient conditions for identifying FSSC nodes in general graphs. 
This paper also proposes a method to \textcolor{black}{exactly determine the dimension of the Strongly Structurally Controllable Subspace (SSCS)} in hierarchical directed acyclic graphs, 
using a combination of graph-theoretical approaches and controllability matrix analyses. 
This method facilitates the identification of FSSC nodes and enhances our understanding of the robustness of node controllability against parameter variations in structured networks.
\end{abstract}



\section{Introduction}
Network controllability, which is the ability to guide a complex system to a desired state, is a significant area of study in system dynamics and control theory. 
\textcolor{black}{Within the context of structured networks \cite{jia2020unifying}, the \textit{structured network} of a system, represented by $\dot{x}=Ax+Bu$, is characterized by the set of patterns $\{*,0,?\}$ in the system matrix $A$.
Here, $*$ denotes non-zero patterns, $0$ represents fixed zeros, and $?$ indicates arbitrary patterns that can take either non-zero values or remain fixed as zero.
Since most of the literature on controllability 
\cite{lin1974structural,mayeda1979strong,jarczyk2011strong,commault2017fixed,park2023kalman} 
focuses on structured networks defined by the non-zero/zero patterns $\{*,0\}$, this paper adopts the same framework. 
In this context, the variations in the values assigned to $*$ are referred to as the \textit{network parameters} of structured networks.
}
The controllability of these structured networks was first established by Lin \cite{lin1974structural} under the name of \textit{structural controllability},
which refers to the controllability of a structured network under almost all network parameters. 
The authors in \cite{dion2003generic} showed that if a structured network is controllable for a specific realization of network parameters, then it is controllable \textcolor{black}{for almost all network parameters.} 
This characteristic is known as the \textit{generic property} of structured networks, meaning that a structurally controllable network remains controllable under almost all network parameters.
However, it also implies that a structurally controllable network could become uncontrollable \textcolor{black}{for a certain realization of network parameters.}
To resolve this issue, the concept of \textit{strong structural controllability} \cite{mayeda1979strong} has been introduced. 
This approach provides a more robust framework compared to structural controllability, ensuring the controllability \textcolor{black}{for all realizations of network parameters.}
Furthermore, some studies consider topological perturbations in networks. For instance, \cite{monshizadeh2014zero,van2017distance} explore the perturbations of self-loops,
and \cite{mousavi2020strong} examines strong structural controllability's robustness to edge modifications for directed acyclic graphs.

For a deeper comprehension of controllability in structured networks, exploring the dimensions of controllable subspaces is important.
Within structured networks, the dimension of controllable subspace fluctuates based on network parameters, 
bounded by two extremes \cite{zhang2013upper}: the maximum dimension, represented by the dimension of \textit{Structurally Controllable Subspace (SCS)}, 
and the minimum dimension, indicated by the dimension of \textit{Strongly Structurally Controllable Subspace (SSCS)}.
In the existing literature, while the dimension of SCS is precisely determined by Hosoes's theory in \cite{hosoe1980determination}, the dimension of SSCS is still not fully explored. 
\textcolor{black}{
Most of the existing results primarily rely on the lower bounds of the dimension of SSCS for specific dynamics, such as Laplacian dynamics 
\cite{yaziciouglu2020strong,yazicioglu2012tight,yaziciouglu2016graph}.
More recently, research has been extended to finding the lower bounds of the dimension of SSCS for general dynamics as well \cite{monshizadeh2015strong,park2023kalman}. 
These studies employ various approaches, such as the derived set \cite{monshizadeh2015strong,yaziciouglu2020strong}, 
graph distance \cite{yazicioglu2012tight,yaziciouglu2016graph}.
\textcolor{black}{However, there is still no work that provides an exact determination of the dimension of SSCS, leaving it as an open problem.}
}

On the other hand, even if the dimension of the controllable subspace remains fixed, countless realizations of the controllable subspace may exist due to variations in network parameters.
Therefore, a state that is controllable in specific realizations of the controllable subspace might be uncontrollable in another.
To address this issue, the concept of the \textit{fixed controllable subspace} was first introduced by the authors in \cite{commault2017fixed}. 
This subspace is defined as the intersection of all possible controllable subspaces that change with the network parameters. 
More specifically, the concept they introduced is the \textit{Fixed Structurally Controllable Subspace (FSCS)}, which is based on structural controllability.
\textcolor{black}{Furthermore, they interpreted the FSCS from a graph point of view, which is the structural controllability of individual state nodes, referred to as \textit{Fixed Structurally Controllable (FSC) nodes}.
However, since the concept of FSCS is fundamentally based on structural controllability, which has the \textit{generic property}, 
such FSC nodes still carry control uncertainties under specific realizations of network parameters.}
To eliminate such uncertainties, it is essential to apply the concept of fixed controllable subspace within the framework of strong structural controllability.

This paper introduces the concept of the \textit{Fixed Strongly Structurally Controllable Subspace (FSSCS)}, grounded in the strong structural controllability.
From a graph-theoretical perspective, we also define \textit{Fixed Strongly Structurally Controllable (FSSC)} nodes and provide their necessary and sufficient conditions for general graphs.
\textcolor{black}{FSSC nodes remain controllable across all realizations of network parameters.
Understanding FSC and FSSC nodes not only enhances our ability to analyze control resilience against parameter variations but also offers practical applications in task assignment problems, 
where the criticality of individual nodes can inform the allocation of essential tasks.}
To achieve this, we introduce a novel approach that merges graph-theoretical insights with controllability matrix interpretations, 
enabling the determination of the dimension of SSCS in hierarchical directed acyclic graphs (HDAG).
\textcolor{black}{Hierarchical structures are actively studied in the field of Hierarchical Multi-Agent Systems (HMAS), 
such as in optimization \cite{roshanzamir2020new}, energy systems \cite{cai2011hierarchical}, and system resilience \cite{rieger2013hierarchical}, showcasing their versatility in modeling real-world hierarchical networks.
For further details and illustrative examples, please refer to the Supplementary Materials, which are referenced accordingly throughout this paper.}
The contributions of our paper can be summarized as follows:
\begin{itemize}
\item This paper extends the results of \cite{commault2017fixed} to strong structural controllability.
By introducing the concept of FSSCS for the first time, we provide the necessary and sufficient conditions for determining FSSC nodes in general graphs.
\item While \cite{yazicioglu2012tight,yaziciouglu2016graph,monshizadeh2015strong} focus on lower bounds for the dimension of SSCS, 
\textcolor{black}{this paper proposes a method for the exact determination of the dimension of SSCS for HDAGs.}
\end{itemize}

This paper is structured as follows:
\textit{Section~\ref{sec_pre}} covers foundational concepts and challenges in structured networks.
\textit{Section~\ref{sec_FSSCS}} introduces FSSCS.
\textit{Section~\ref{sec_cont}} explores the graph-theoretical meaning of the controllability matrix.
\textit{Section~\ref{sec_dim}} details the dimension of SSCS in HDAGs.
\textit{Section~\ref{sec_ex}} and \textit{Section~\ref{sec_conc}} present examples and conclusion, respectively.


\section{Preliminaries} \label{sec_pre}
Let us consider a network with $n$-states and $m$-inputs:
\begin{align} \label{eq_sys}
\dot{x} = \mathcal{A}x + \mathcal{B} u,
\end{align}
where $\mathcal{A} \in \mathbb{R}^{n \times n}$ is the system matrix and \textcolor{black}{$\mathcal{B} \in \mathbb{R}^{n \times m}$ is the input matrix consisting of $m$-standard basis vectors.}
\textcolor{black}{The network is said to be controllable if, for any initial state $x_0$ and any final state $x_f$, there exists a control input $u$ that can steer the system from $x_0$ to $x_f$ within a finite time interval.}
A condition for controllability is that the controllability matrix 
\begin{equation}\label{eq_cont}
\mathcal{C} = [\mathcal{B}, \mathcal{A}\mathcal{B}, \mathcal{A}^2\mathcal{B}, \cdots, \mathcal{A}^{n-1}\mathcal{B}] \in\mathbb{R}^{n \times nm}, 
\end{equation}
has full rank, i.e., $\text{rank}(\mathcal{C}) = n$.
\textcolor{black}{
The parameterized matrix $\mathcal{A}_\mathcal{P}$, reflecting the non-zero/zero patterns of $\mathcal{A}$, is defined as 
$\mathcal{A}_\mathcal{P} \in \mathbb{R}^{n \times n}$, where $[\mathcal{A}_\mathcal{P}]_{i,j} = 0$ if $[\mathcal{A}]_{i,j} = 0$ and 
$[\mathcal{A}_\mathcal{P}]_{i,j}= a_{ij}$ if $[\mathcal{A}]_{i,j} \neq 0$.
Here, the non-zero elements $a_{ij}$ can take any value except zero, which are called the \textit{network parameters}.
The structured network of \eqref{eq_sys} can be represented as:
\begin{align} \label{eq_sys_st}
\dot{x} = \mathcal{A}_\mathcal{P}x + \mathcal{B} u.
\end{align}}
The structured network in \eqref{eq_sys_st} can be represented as a graph:
\begin{equation} \label{eq_sys_graph}
\mathcal{G}(\mathcal{V},\mathcal{E}),
\end{equation}
where $\mathcal{V}$ denotes the set of state nodes, and $\mathcal{E}$ represents the set of edges. 
If $[\mathcal{A}_\mathcal{P}]_{j,i} \neq 0$, then there exists an edge $(i,j) \in \mathcal{E}$ with a non-zero weight, where $i, j \in \mathcal{V}$.
Furthermore, a state node $i \in \mathcal{V}$ connected to an input $u$ with an edge $(u, i) \in \mathcal{E}$, is referred to as a \textit{leader}. 
The set of such leaders within $\mathcal{G}(\mathcal{V},\mathcal{E})$ is represented as $\mathcal{V}_{\mathcal{L}} \subseteq \mathcal{V}$.
\textcolor{black}{
A \textit{path} from state node $i_1$ to state node $i_p$ is defined as a sequence of connected edges $(i_1, i_2), (i_2, i_3), \ldots, (i_{p-1}, i_p)$,
where $i_k \in \mathcal{V}$ for $k \in \{1, \ldots, p\}$, and $(i_{k}, i_{k+1}) \in \mathcal{E}$ for $k \in \{1, \ldots, p-1\}$, with no node revisited along the path.}
This path can also be represented as a node sequence denoted by $(i_1 \rightarrow i_2 \rightarrow \ldots \rightarrow i_p)$. 
A path consisting of $k$ edges is said to have $k$-steps, thus defining the length of the path in terms of the number of transitions between nodes.
A path that initiates from a leader, i.e., $i_1 \in \mathcal{V}_{\mathcal{L}}$, is called a \textit{stem}, and a path where $i_1 = i_p$ is referred to as a \textit{cycle}.
A graph without any cycles is classified as \textit{Directed Acyclic Graph (DAG)}.
We say that a graph $\mathcal{G}(\mathcal{V},\mathcal{E})$ is \textit{input-connected} if there exists a stem leading to each node in $\mathcal{V}$. 
Input-connectedness ensures that the influence of inputs (or leaders) propagates through the entire network, affecting all state nodes. Therefore, input-connectedness is a necessary condition for controllability.
For a graph $\mathcal{G}(\mathcal{V},\mathcal{E})$ with a leader, 
consider $\mathcal{V}_k$ as the set of nodes that can be reached from the leader with $k$-steps, for $k \in \{1, \ldots, p\}$,
where $p$ is the number of steps required to reach the node that is the farthest from the leader.
If the graph is input-connected, then the entire set of nodes $\mathcal{V}$ is the union of these reachable sets, i.e., $\mathcal{V} = \bigcup_{k=1}^p \mathcal{V}_k$. 
A graph is a \textit{Hierarchical Directed Acyclic Graph (HDAG)} if the sets $\mathcal{V}_k$ are mutually disjoint for $k \in \{1, \ldots, p\}$.

\begin{remark}
\textcolor{black}{
From \cite{liu2012control,park2024fixed}, every input-connected DAG can be represented as a hierarchical structure with distinct layers, where edges are directed from higher to lower layers. 
A key characteristic of DAGs' hierarchical structure is that it allows edges between non-adjacent layers, enabling flexible interactions across multiple levels. 
In contrast, HDAGs restrict edges to only connect consecutive layers, enforcing a top-down hierarchical flow. 
While HDAGs retain most of the properties of DAGs, they ensure a well-defined hierarchical structure between layers.
For more details, refer to \textit{Section II} in the Supplementary Materials.
}
\end{remark}

The structured network given by \eqref{eq_sys_st} achieves \textit{structural controllability} when it is controllable under almost all numerical realizations of network parameters.
Furthermore, it reaches \textit{strong structural controllability} when it remains controllable under \textcolor{black}{all numerical realizations of network parameters.}
For detailed examples illustrating these two concepts of controllability with respect to network parameters, check \textit{Example 1} in the Supplementary Materials.

As shown in \cite{commault2017fixed}, the rank of the controllability matrix for a structured network can vary depending on the network parameters.
This variation in rank directly influences the dimension of controllable subspace, which is the column space of the controllability matrix.
To describe these variations, we introduce two key concepts: the \textit{Structurally Controllable Subspace (SCS)} \cite{commault2017fixed} and the \textit{Strongly Structurally Controllable Subspace (SSCS)}.
\textcolor{black}{
The SCS is the controllable subspace with the maximum dimension, while the SSCS is the controllable subspace with the minimum dimension. 
Note that controllable subspaces with the same dimension can exist in countless realizations depending on the network parameters. 
Therefore, there can be countless realizations of both SCS and SSCS. 
Despite these realizations, we denote the SCS as $\mathcal{C}_\Lambda$ and the SSCS as $\mathcal{C}_\Gamma$ to provide a consistent notation with \cite{commault2017fixed}.}
For a structured network represented by \eqref{eq_sys_st}, the dimension of controllable subspace (or rank of the controllability matrix) is constrained by the following boundary condition:
\begin{align} \label{ex_1_dim}
|\mathcal{C}_\Gamma| \leq {rank}(\mathcal{C}) \leq |\mathcal{C}_\Lambda| \leq n,
\end{align}
where $|\mathcal{C}_\Lambda|$ and $|\mathcal{C}_\Gamma|$ represent the dimensions of the respective controllable subspaces.
In the framework of structured networks, even if the dimension of the controllable subspace remains fixed, 
the subspace itself can still change due to variations in network parameters. 
This means that a state that is controllable within one realization of the controllable subspace may become uncontrollable in another.
To address this issue, the concept of a fixed controllable subspace becomes crucial.
For a deeper understanding of the fixed controllable subspace, check \textit{Example 2} in the Supplementary Materials.

\subsection{Fixed Structurally Controllable Subspace (FSCS)}
In the context of structural controllability, the authors in \cite{commault2017fixed} first introduced the notion of \textit{Fixed Structurally Controllable Subspace (FSCS)}, which is the intersection of all realizations of SCS.
This subspace represents a consistent controllable subspace for almost all network parameters.
The relationship between SCS and FSCS can be established as follows \cite{commault2017fixed}:
\begin{equation} \label{eq_fixed1}
\mathcal{C}^\mathcal{F}_\Lambda \subseteq \mathcal{C}_{{\Lambda}} \subseteq \mathbb{R}^n,
\end{equation}
where $\mathcal{C}_{{\Lambda}}$ and $\mathcal{C}_\Lambda^\mathcal{F}$ are the SCS and FSCS, respectively.

Drawing on graph theory, we delve into the concept of \textit{Fixed Structurally Controllable (FSC) nodes}, as introduced in \cite{commault2017fixed}. 
For a graph $\mathcal{G}(\mathcal{V},\mathcal{E})$, a state node $i \in \mathcal{V}$ is determined as an FSC node 
if its corresponding vector $v_i \in \mathbb{R}^n$ is included within the FSCS, where $v_i$ is the standard basis vector with a non-zero value only in its $i$-th element.
\textcolor{black}{Considering that the FSCS is the intersection of all realizations of SCS, 
a standard basis vector corresponding to a state node being included in the FSCS means that the state node is controllable in all realizations of SCS. 
However, note that since FSC nodes are based on SCSs, which are the controllable subspaces for almost all network parameters, they may become uncontrollable for specific realizations of network parameters.}
\textcolor{black}{Without loss of generality, we assume that the graph is input-connected throughout this paper.
Before introducing the condition for FSC nodes, let us consider the existing result in \cite{hosoe1980determination}:}

\begin{proposition}\label{thm_hosoe}
\textcolor{black}{
\cite{hosoe1980determination}
For a graph $\mathcal{G}(\mathcal{V},\mathcal{E})$, 
the dimension of SCS is the maximum number of state nodes that can be covered by a disjoint set of stems and cycles in $\mathcal{G}(\mathcal{V},\mathcal{E})$.}
\end{proposition}

Note that while the disjoint sets of stems and cycles within a graph may not be uniquely determined, the maximum number of state nodes contained in such a set remains unique.
Based on \textit{Proposition~\ref{thm_hosoe}}, the following proposition in \cite{commault2017fixed} provides the graph-theoretical condition for determining the FSC nodes.

\begin{proposition}\label{thm_commault}\textcolor{black}{
\cite{commault2017fixed}
For a graph $\mathcal{G}(\mathcal{V},\mathcal{E})$, 
a state node $k\in\mathcal{V}$ is an FSC node if and only if $k$ becoming a leader with additional input 
does not increase the dimension of SCS.}
\end{proposition}

\textcolor{black}{
Note that \textit{Proposition~\ref{thm_hosoe}} can be solved using the maximum matching algorithm \cite{hopcroft1973n}, which has polynomial complexity. Consequently, \textit{Proposition~\ref{thm_commault}} also has polynomial complexity.
For a detailed complexity analysis, please refer to \textit{Section V.A} in our Supplementary Materials.}

\section{Fixed Strongly Structurally Controllable Subspace}\label{sec_FSSCS}
In the context of strong structural controllability, this section first introduces the \textit{Fixed Strongly Structurally Controllable Subspace (FSSCS)},
an extension of the FSCS concept presented in \cite{commault2017fixed}.
From a graph-theoretical perspective, we provide the necessary and sufficient conditions for identifying \textit{Fixed Strongly Structurally Controllable (FSSC)} nodes in general graphs.
The FSSCS is determined by the intersection of all realizations of SSCS, which offers the most robust framework of controllable subspace within structured networks, 
ensuring controllability across all network parameters.
The relationship between SSCS and FSSCS can be established as follows:
\begin{equation} \label{eq_fixed1}
\mathcal{C}^\mathcal{F}_\Gamma \subseteq \mathcal{C}_{{\Gamma}} \subseteq \mathbb{R}^n,
\end{equation}
where $\mathcal{C}_{\Gamma}$ and $\mathcal{C}_\Gamma^\mathcal{F}$ are the SSCS and FSSCS, respectively.
The relationship between the FSCS and FSSCS is given by:
\begin{equation} \label{eq_fixed2}
\mathcal{C}^\mathcal{F}_\Gamma \subseteq \mathcal{C}^\mathcal{F}_\Lambda \subseteq \mathbb{R}^n.
\end{equation}
Note that when the dimensions of SCS and SSCS coincide, the FSCS and FSSCS form identical controllable subspaces. 
For the necessity of the FSSCS compared to FSCS, please refer to \textit{Example 4} in the Supplementary Materials.

From a graph-theoretical perspective, we introduce the concept of FSSC nodes, which is a stronger concept than FSC nodes.
For a graph $\mathcal{G}(\mathcal{V},\mathcal{E})$, we define a state node $i \in \mathcal{V}$ as 
an FSSC node if its corresponding standard basis vector $v_i \in \mathbb{R}^n$ is a part of the FSSCS.
To further understand this, let us consider the SSCS, denoted as $C_{\Gamma}$, which corresponds to the pair $(\mathcal{A}_\mathcal{P},\mathcal{B})$ in \eqref{eq_sys_st}. 
To analyze how the inclusion of a state node $i$ as an additional leader affects the SSCS,
we introduce an augmented input matrix defined as ${\mathcal{B}}_{v_i}=[\mathcal{B},v_i]$.
For an augmented pair $(\mathcal{A}_\mathcal{P}, \mathcal{B}_{v_i})$, the corresponding SSCS is:
\begin{equation}\label{eq}
C_{\Gamma v_i}= \text{Column space}[{\mathcal{B}}_{v_i}, \mathcal{A}_\mathcal{P}{\mathcal{B}}_{v_i}, ..., \mathcal{A}_\mathcal{P}^{n-1}\mathcal{B}_{v_i}].
\end{equation} 

The following theorem provides the necessary and sufficient conditions for FSSC nodes in a general graph.

\begin{theorem}\label{thm_FSSC}
For a graph $\mathcal{G}(\mathcal{V},\mathcal{E})$, 
a state node $i\in\mathcal{V}$ is an FSSC node if and only if $i$ becoming a leader with additional input 
does not increase the dimension of SSCS.
\end{theorem}
\begin{proof}
We begin by establishing the necessary condition for a node to be included in the FSSCS.
For a structured network represented by the pair $(\mathcal{A}_\mathcal{P},\mathcal{B})$, let $v_i\in\mathbb{R}^{n}$ be a standard basis vector. 
Then $v_i\in\mathcal{C}_\Gamma^\mathcal{F}$ if and only if $C_{\Gamma v_i}=C_{\Gamma}$ for all network parameters.
For \textit{only if} condition, it is clear that $C_{\Gamma} \subseteq C_{\Gamma v_i}$. 
Assume $v_i$ is included in $\mathcal{C}_\Gamma^\mathcal{F}$ and thus in $C_{\Gamma}$. 
Therefore, $\mathcal{A}_\mathcal{P}^j v_i$ belongs to $C_{\Gamma}$ for each $j \in \{1, \ldots, n-1\}$. 
This implies that $v_i$, along with all its subsequent states generated by the powers of $\mathcal{A}_\mathcal{P}$, is contained in $C_{\Gamma}$ as described in \eqref{eq}.
This means all vectors in $C_{\Gamma v_i}$ are also contained within $C_{\Gamma}$, 
thereby establishing that $C_{\Gamma v_i}$ is equal to $C_{\Gamma}$ if $v_i\in\mathcal{C}_\Gamma^\mathcal{F}$.
For \textit{if} condition, let us assume $C_{\Gamma} = C_{\Gamma v_i}$. By the definition of controllability, $v_i$ exists within $C_{\Gamma v_i}$ for all network parameters. 
Since $C_{\Gamma v_i}$ is equivalent to $C_{\Gamma}$, $v_i$ must be included in $C_{\Gamma}$, meaning $v_i$ is in $\mathcal{C}_\Gamma^\mathcal{F}$.
Therefore, the necessary and sufficient condition for a standard basis vector $v_i\in\mathbb{R}^{n}$ to be included in the FSSCS 
is that the dimension of SSCS for the augmented pair $(\mathcal{A}_\mathcal{P}, \mathcal{B}_{v_i})$ 
remains consistent with that of the original pair $(\mathcal{A}_\mathcal{P}, \mathcal{B})$, where ${\mathcal{B}}_{v_i}=[\mathcal{B},v_i]$. 
From a graph-theoretical perspective, this implies that a state node $i \in \mathcal{V}$ is an FSSC node if and only if 
the dimension of SSCS does not increase when $i$ becomes a leader with additional input.
\end{proof}

Note that an FSSC node is controllable in all realizations of SSCS. 
Therefore, it is also a sufficient condition for being an FSC node, which is controllable in all realizations of SCS. 
Consequently, an FSSC node represents the strongest concept of controllability for individual nodes, ensuring that the state node is controllable for all realizations of network parameters.
From \textit{Theorem~\ref{thm_FSSC}}, determining an FSSC node requires knowing the dimension of SSCS, i.e., $|\mathcal{C}_\Gamma|$ in \eqref{ex_1_dim}.
\textcolor{black}{However, due to the difficulty in determining this dimension of SSCS, 
most studies focus on establishing its lower bounds \cite{monshizadeh2015strong,yaziciouglu2020strong,yazicioglu2012tight,yaziciouglu2016graph}.}

\begin{remark}\textcolor{black}{
From the proof of \textit{Theorem~\ref{thm_FSSC}}, it follows that all standard basis vectors forming the FSSCS have a one-to-one correspondence with the FSSC nodes; that is, 
if a node is identified as an FSSC node, its corresponding standard basis vector is guaranteed to be in the FSSCS, and conversely, any standard basis vector in the FSSCS corresponds to an FSSC node. 
However, this does not imply that the FSSCS is generated only by those standard basis vectors. 
In fact, even a non-standard vector having more than two non-zero elements may serve as the basis for the FSSCS. 
However, such a vector inherently embeds a dependency among at least two nodes, making it impossible to control each node independently for all network parameters. 
Since this paper focuses on identifying which nodes can be independently controlled, i.e., FSSC nodes, considering only the standard basis vectors that constitute the FSSCS is sufficient.
For a deeper understanding, please refer to \textit{Section II} in the Supplementary Materials.}
\end{remark}

\section{Graph-Theoretical Meaning of Controllability Matrix}\label{sec_cont}
This section focuses on the graph-theoretical interpretation of elements in the controllability matrix, 
which serves as the foundational background for exploring the dimension of SSCS in the following sections.
For a graph $\mathcal{G}(\mathcal{V},\mathcal{E})$, let us consider a stem from a leader to $i\in\mathcal{V}$ with $k$-steps, 
denoted as $\mathcal{S}^{i,k}_{p}$, where $p$ is the index of stems.
Let us denote the weight of $w$-th edge in the sequence of $\mathcal{S}^{i,k}_{p}$ as $e^w_p$  where $w \in \{1, ..., k\}$.
Then, the \textit{weight product (WP)} of $\mathcal{S}^{i,k}_{p}$, is defined as:
\begin{equation}\label{eq_gain}
\mathcal{WP}(\mathcal{S}^{i,k}_{p}) = \prod_{w=1}^{k} e^w_p, 
\end{equation}
which reflects the product of the edge weights in $p$-th stem. 
If there are $m$-stems to a state node $i \in \mathcal{V}$ with $k$-steps, 
we denote the set of these $m$-stems as $\mathcal{S}^{i,k}_{\Sigma}$.
The \textit{sum of weight products (SWP)} for all stems to $i \in \mathcal{V}$ is defined as:
\begin{equation}\label{eq_gain2}
\mathcal{SWP}(\mathcal{S}^{i,k}_{\Sigma}) = \sum_{p=1}^{m} \mathcal{WP}(\mathcal{S}^{i,k}_{p}), %
\end{equation}
where $m$ is the total number of such stems to node $i$ with $k$-steps.
From \cite{park2024controllability}, it is known that the $ik$-th element in the controllability matrix given by \eqref{eq_cont} implies
the SWP of all stems from a leader to $i\in\mathcal{V}$ with $k$-steps.
Therefore, for the controllability matrix $\mathcal{C}\in\mathbb{R}^{n \times n}$ with a leader, we have:
\begin{equation}\label{eq_cont_meaning}
[\mathcal{C}]_{i,k+1} = \mathcal{SWP}(\mathcal{S}^{i,k}_{\Sigma}), %
\end{equation}
where $k\in\{0,...,n-1\}$ and $i\in\{1,...,n\}$.
\textcolor{black}{
For a detailed example illustrating the SWP concept within the controllability matrix, please refer to \textit{Example 5} in the Supplementary Materials.}
Note that the interpretation based on SWP can be extended to controllability matrices with multiple leaders \cite{park2024controllability}.
For further analysis, we classify the non-zero elements in the controllability matrix into \textit{single-term} and \textit{multi-term}.

\begin{definition}\label{def_term}\cite{park2024controllability} 
For the controllability matrix in \eqref{eq_cont}, the non-zero elements can be classified as follows:
\begin{itemize}
\item An element is a \textbf{single-term} if it consists of a product of non-zero edge weights or is a single edge weight.
\item An element is a \textbf{multi-term} if it is composed of the sum of two or more \textit{single-terms}.
\end{itemize}
\end{definition}

When analyzing the dimension of SSCS (or the minimum rank of the controllability matrix), it is crucial to consider non-zero elements that have the potential to become zero.
This definition highlights a critical distinction: \textit{Single-term} in the controllability matrix is invariably non-zero, whereas \textit{multi-term} may potentially be zero. 
Given this observation, the presence of \textit{multi-term} causes uncertainty in determining the rank of controllability matrix. 
\textcolor{black}{
Considering that SWP is defined as the sum of WPs in \eqref{eq_cont_meaning}, 
the graph-theoretical interpretations of \textit{single-terms} and \textit{multi-terms} within the controllability matrix can be established as follows:
\begin{proposition}\label{prop_term}
For a graph $\mathcal{G}(\mathcal{V},\mathcal{E})$ with a single leader, 
the element $[\mathcal{C}]_{i,k+1}$ in the controllability matrix $\mathcal{C} \in \mathbb{R}^{n \times n}$ is:
\begin{itemize}
    \item A \textbf{single-term} if and only if there exists a unique stem leading to node $i\in\mathcal{V}$ with $k$-steps.
    \item A \textbf{multi-term} if and only if multiple stems are leading to node $i\in\mathcal{V}$ with $k$-steps.
\end{itemize}
\end{proposition}}
For convenience, the unique element in the $i$-th row of the controllability matrix will be referred to as the \textit{single-term} (or \textit{multi-term}) of a node $i \in \mathcal{V}$.
\textcolor{black}{To maintain clarity and consistency in the following sections, we assume that:
}

\begin{assumption}\label{assum}
A graph is an HDAG with a single leader.
\end{assumption}
This assumption focuses on HDAGs, where all stems to each state node have the same number of steps.
It implies that each row of the controllability matrix is either a zero vector or contains exactly one non-zero element.
Note that while the results presented from this point forward are restricted to a single leader, 
they can be extended to multiple leaders, as shown in \cite{park2024controllability,park2024fixed}. 

For a graph $\mathcal{G}(\mathcal{V},\mathcal{E})$, we define nodes within $\mathcal{V}$ that have multiple incoming edges as \textit{integrators}, 
with the set of integrators denoted by $\mathcal{V}^{int}$.
Additionally, we introduce the concept of \textit{intermediators}, defined as nodes within $\mathcal{V}$ that have exactly one incoming edge,
where all stems leading to these nodes must include at least one integrator.
The set of intermediators is denoted by $\mathcal{V}^{med}$.
Based on these definitions, the following lemma provides a condition for the occurrence of \textit{multi-terms} in the controllability matrix.

\begin{lemma}\label{lem_int_med}
For a graph $\mathcal{G}(\mathcal{V},\mathcal{E})$,
a node $i\in\mathcal{V}$ is included in $\mathcal{V}^{int}\cup\mathcal{V}^{med}$
if and only if the $i$-th row of the controllability matrix has a \textit{multi-term}.
\end{lemma}
\begin{proof}
For \textit{only if} condition, let us suppose $i \in (\mathcal{V}^{int} \cup \mathcal{V}^{med})$. 
This implies that $i$ either has multiple incoming edges or is reached by stems that include at least one integrator. 
In both cases, multiple stems leading to $i$ are necessitated. 
Consequently, by \textit{Proposition~\ref{prop_term}}, the $i$-th row of the controllability matrix must contain a \textit{multi-term}.
For \textit{if} condition, let us suppose that the $i$-th row in the controllability matrix contains a \textit{multi-term}. 
From \textit{Proposition~\ref{prop_term}} and \textit{Assumption~\ref{assum}}, there are multiple stems leading to $i$ with the same step. 
If these multiple stems are edge disjointed, $i$ must have multiple incoming edges, thus $i \in \mathcal{V}^{int}$. 
If these stems are not edge disjointed, all multiple stems leading to $i$ include at least one integrator, 
making either $i \in \mathcal{V}^{int}$ or $i \in \mathcal{V}^{med}$.
\end{proof}

Note that although \textit{multi-terms} have the potential to become zero, not every \textit{multi-term} can always be zero due to dependencies on network parameters.
\textcolor{black}{
For an example illustrating this, check \textit{Example 6} in the Supplementary Materials.}
For a given controllability matrix, the following theorem provides the condition that at least one \textit{multi-term} cannot be zero.

\begin{theorem}\label{thm_integrator}
For a graph $\mathcal{G}(\mathcal{V},\mathcal{E})$,
at least one \textit{multi-term} in the controllability matrix cannot be zero if and only if
there exists $i_w\in \mathcal{V}^{int}$ for which there is exactly one stem leading to $i_w$ that does not include any other integrator.
\end{theorem}

\begin{proof}
For \textit{if} condition, suppose that there exists exactly one stem leading to an integrator $i_w \in \mathcal{V}^{int}$, which does not include any other integrator. 
Let the WP of this unique stem be denoted as $\alpha_{i_w}$. 
Given that integrators have at least two incoming edges, there must be at least one other stem leading to $i_w$ that includes another integrator $i_l \in \mathcal{V}^{int}$, for $w \neq l$. 
Let these stems be denoted as $\mathcal{S}^{i_w}_p$ for $p \in \{1,...,m\}$, and the SWP of these $m$-stems be denoted as $\beta_{i_w}$.
Thus, the number of stems leading to $i_w$ is $m+1$.
Furthermore, let the SWP of all stems leading to $i_l$ be denoted as $\beta_{i_l}$.
Then, the \textit{multi-term} of $i_w$ is represented as $\alpha_{i_w} + \beta_{i_w}$,
where $\beta_{i_w}$ is obtained by multiplying $\beta_{i_l}$ with the WP of the path from $i_l$ to $i_w$. 
In this case, if we set the \textit{multi-term} of $i_l$ to zero, i,e, $\beta_{i_l}=0$, the \textit{multi-term} of $i_w$ becomes \textit{single-term}, which can not invariably be zero.
Consequently, this indicates that the \textit{multi-term} of $i_w$ cannot be zero.
For the \textit{only if} condition, suppose that at least one \textit{multi-term} in the controllability matrix cannot be zero. 
For the sake of contradiction, assume that there is no integrator $i_w \in \mathcal{V}^{int}$ which has exactly one stem leading to it that does not include any other integrator. 
Then, each integrator $i_w \in \mathcal{V}^{int}$ either has no or at least two stems that do not include another integrator.
To establish a contradiction, we will show that under the following two cases, all \textit{multi-terms} could be zero, which contradicts our initial assumption.

\begin{itemize}
\item \textbf{Case 1:} An integrator $i_w \in \mathcal{V}^{int}$ has multiple stems that exclude any other integrator.
\item \textbf{Case 2:} An integrator $i_w \in \mathcal{V}^{int}$ does not have any stem that excludes any other integrator. 
\end{itemize}

In \textbf{Case 1}, let us denote the SWP of multiple stems as $\alpha_{i_w}$, and the SWP of stems that include at least one other integrator $i_l \in \mathcal{V}^{int}$, as $\beta_{i_w}$. 
Moreover, let $\beta_{i_l}$ represent the \textit{multi-term} of $i_l$. 
From \textit{Assumption~\ref{assum}}, multiple stems leading to $i_w$ must have the same steps, 
which means the \textit{multi-term} of $i_w$ is represented as $\alpha_{i_w} + \beta_{i_w}$, 
where $\beta_{i_w}$ is obtained by multiplying $\beta_{i_l}$ with the WP of the path from $i_l$ to $i_w$.
Given that $\alpha_{i_w}$ is the sum of two or more \textit{single-terms}, 
both $\alpha_{i_w}$ and all instances of $\beta_{i_l}$ included in $\beta_{i_w}$ have the potential to be zero. 
Therefore, the \textit{multi-term} of $i_w$ can always be zero.
In \textbf{Case 2}, it follows that every stem leading to each integrator includes at least one integrator. 
From {Case 1}, the integrators included in each stem can be zero.
Since the \textit{multi-term} of $i_w$ is the sum of these, it also can be zero.
Furthermore, according to \textit{Lemma~\ref{lem_int_med}}, intermediators $j_w \in \mathcal{V}^{med}$ can also have \textit{multi-terms}, 
but each intermediator has exactly one stem that includes at least one integrator $i_w \in \mathcal{V}^{int}$. 
Therefore, the \textit{multi-term} of $j_w$ is determined by the product of the \textit{multi-term} of $i_w$ and the WP of the path from $i_w$ to $j_w$.
Given the outcomes of Cases 1 and 2, where the \textit{multi-terms} of such integrators are shown to potentially be zero, it follows that the \textit{multi-term} of $j_w$ can similarly be zero.
Therefore, if each integrator $i_w \in \mathcal{V}^{int}$ either does not have a stem that excludes other integrators or has more than one, 
all \textit{multi-terms} can be zero, thereby contradicting our initial assumption. 
\end{proof}

The above theorem provides the unique condition under which at least one \textit{multi-term} in the controllability matrix cannot become zero.
Subsequent discussions on the dimension of SSCS will be based on the condition provided in \textit{Theorem~\ref{thm_integrator}}.

\begin{figure}[t]
\centering
\subfigure[$\mathcal{G}(\mathcal{V},\mathcal{E})$]{
\begin{tikzpicture}[scale=0.7]
\draw[dashed, gray] (-1.7,3) -- (2.7,3) node[left, pos=1.29] { \scriptsize$0$-step};
\draw[dashed, gray] (-1.7,2) -- (2.7,2) node[left, pos=1.29] { \scriptsize$1$-step};
\draw[dashed, gray] (-1.7,1) -- (2.7,1) node[left, pos=1.29] { \scriptsize$2$-step};
\draw[dashed, gray] (-1.7,0) -- (2.7,0) node[left, pos=1.29] { \scriptsize$3$-step};

\node[] at (1.75,3.75) {$u$};
\node[] (node0) at (1.75,3.75) [] {};

\node[] at (1,1.6) {\scriptsize$i_l$};
\node[] at (0,0.6) {\scriptsize$i_w$};

\node[place, circle,minimum size=0.5cm] (node1) at (1,3) [] {\scriptsize$1$};
\node[place, circle,minimum size=0.5cm] (node2) at (0,2) [] {\scriptsize$2$};
\node[place, circle,minimum size=0.5cm] (node3) at (2,2) [] {\scriptsize$3$};
\node[place, circle,minimum size=0.5cm] (node4) at (-1,1) [] {\scriptsize$4$};
\node[place, circle,minimum size=0.5cm] (node5) at (1,1) [] {\scriptsize$5$};
\node[place, circle,minimum size=0.5cm] (node6) at (0,0) [] {\scriptsize$6$};

\draw (node0) [-latex, line width=0.5pt] -- node [right]  {} (node1);
\draw (node1) [-latex, line width=0.5pt] -- node [right]  {} (node2);
\draw (node1) [-latex, line width=0.5pt] -- node [right]  {} (node3);
\draw (node2) [-latex, line width=0.5pt] -- node [right]  {} (node4);
\draw (node2) [-latex, line width=0.5pt] -- node [right]  {} (node5);
\draw (node3) [-latex, line width=0.5pt] -- node [right]  {} (node5);
\draw (node4) [-latex, line width=0.5pt] -- node [right]  {} (node6);
\draw (node5) [-latex, line width=0.5pt] -- node [right]  {} (node6);
\end{tikzpicture}
}
\subfigure[$\mathcal{G}(\mathcal{V},\mathcal{E})$]{
\begin{tikzpicture}[scale=0.7]
\draw[dashed, gray] (-1.7,3) -- (2.7,3) node[left, pos=1.29] { \scriptsize$0$-step};
\draw[dashed, gray] (-1.7,2) -- (2.7,2) node[left, pos=1.29] { \scriptsize$1$-step};
\draw[dashed, gray] (-1.7,1) -- (2.7,1) node[left, pos=1.29] { \scriptsize$2$-step};
\draw[dashed, gray] (-1.7,0) -- (2.7,0) node[left, pos=1.29] { \scriptsize$3$-step};

\node[] at (1.75,3.75) {$u$};
\node[] (node0) at (1.75,3.75) [] {};

\node[] at (1,1.6) {\scriptsize$i_l$};
\node[] at (0,0.6) {\scriptsize$i_w$};
\node[] at (2,0.6) {\scriptsize$j_l$};

\node[place, circle,minimum size=0.5cm] (node1) at (1,3) [] {\scriptsize$1$};
\node[place, circle,minimum size=0.5cm] (node2) at (0,2) [] {\scriptsize$2$};
\node[place, circle,minimum size=0.5cm] (node3) at (2,2) [] {\scriptsize$3$};
\node[place, circle,minimum size=0.5cm] (node4) at (-1,1) [] {\scriptsize$4$};
\node[place, circle,minimum size=0.5cm] (node5) at (1,1) [] {\scriptsize$5$};
\node[place, circle,minimum size=0.5cm] (node6) at (0,0) [] {\scriptsize$6$};
\node[place, circle,minimum size=0.5cm] (node7) at (2,0) [] {\scriptsize$7$};

\draw (node0) [-latex, line width=0.5pt] -- node [right]  {} (node1);
\draw (node1) [-latex, line width=0.5pt] -- node [right]  {} (node2);
\draw (node1) [-latex, line width=0.5pt] -- node [right]  {} (node3);
\draw (node2) [-latex, line width=0.5pt] -- node [right]  {} (node4);
\draw (node2) [-latex, line width=0.5pt] -- node [right]  {} (node5);
\draw (node3) [-latex, line width=0.5pt] -- node [right]  {} (node5);
\draw (node4) [-latex, line width=0.5pt] -- node [right]  {} (node6);
\draw (node5) [-latex, line width=0.5pt] -- node [right]  {} (node6);
\draw (node5) [-latex, line width=0.5pt] -- node [right]  {} (node7);
\end{tikzpicture}
}
\caption{Graphs satisfying \textit{Condition~\ref{unique_stem_condition}}. In the controllability matrix,
(a) From \textit{Theorem~\ref{thm_integrator}}, either the \textit{multi-term} of $i_w$ or $i_w$ can be zero.
(b) From \textit{Lemma~\ref{lem_mul}}, either the \textit{multi-term} of $i_w$ or the \textit{multi-terms} of $i_l$ and $j_l$ can be zero.
}
\label{fig_cond}
\end{figure}
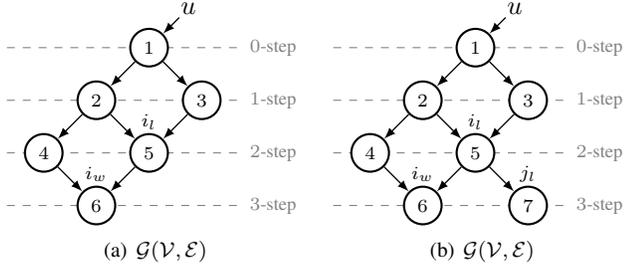

\begin{condition}\label{unique_stem_condition}
For an integrator $i_w \in \mathcal{V}^{int}$ in $\mathcal{G}(\mathcal{V}, \mathcal{E})$, there exists exactly one stem leading to $i_w$ that does not include any other integrator.
By the definition of integrators, this condition ensures the existence of at least one additional stem leading to $i_w$ that includes another integrator $i_l \in \mathcal{V}^{int}$, where $i_w \neq i_l$.
\end{condition}

The simplest graph satisfying this condition is illustrated in Fig.~\ref{fig_cond}(a).
In this graph, the proof of \textit{if} condition in \textit{Theorem~\ref{thm_integrator}} implies that
either the \textit{multi-term} of $i_w$ or $i_l$ cannot be zero.
Note that this is the only case where at least one \textit{multi-term} in the controllability matrix cannot be zero.

\section{The Dimension of Strongly Structurally Controllable Subspace}\label{sec_dim}
In the previous section, we delved into the characteristics of elements constituting the controllability matrix.
Building on our prior explorations, this section seeks to determine the dimension of SSCS, i.e., $|\mathcal{C}_\Gamma|$ in \eqref{ex_1_dim}, 
by examining the influence of \textit{multi-terms} on the maximum rank of controllability matrix.
We first introduce two lemmas based on \textit{Condition~\ref{unique_stem_condition}}.


\begin{lemma}\label{lem_del}
Let us consider an integrator $i_w\in\mathcal{V}^{int}$ satisfying \textit{Condition~\ref{unique_stem_condition}} in $\mathcal{G}(\mathcal{V},\mathcal{E})$.
Then, zeroing the \textit{multi-term} of $i_l \in \mathcal{V}^{int}$, included in a stem leading to $i_w$, does not reduce the maximum rank of controllability matrix.
\end{lemma}
\begin{proof}
Let us consider the unique stem leading to $i_w$ that does not include any other integrator.
From \textit{Assumption~\ref{assum}}, within this stem, there exists a node $i_q \in \mathcal{V}$, 
and the stem leading to $i_q$ has the same steps as the stems leading to $i_l$.
Since the stem leading to $i_q$ excludes any other integrator, within the controllability matrix, the column corresponding to the steps of the stem leading to $i_l$
contains the \textit{single-term} of $i_q$, in addition to the \textit{multi-term} of $i_l$.
Thus, zeroing the \textit{multi-term} of $i_l$ cannot turn this column into a zero vector due to the remaining \textit{single-term} of $i_q$,
implying that the maximum rank of the controllability matrix remains unchanged.
\end{proof}

From \textit{Lemma~\ref{lem_del}}, it follows that even if a particular column consists only of \textit{multi-terms}, it may not necessarily become a zero vector.
\textcolor{black}{
For example, let us consider the following controllability matrix of the graph depicted in Fig.~\ref{fig_cond}(b).
\begin{equation}\label{eq_ex_cont_mul}
\resizebox{0.8\hsize}{!}{$
\mathcal{C} = \begin{bmatrix}
1 & 0 & 0 & 0 & 0 & 0 & 0  \\
0 & a_{12} & 0 & 0 & 0 & 0 & 0  \\
0 & a_{13} & 0 & 0 & 0 & 0 & 0  \\
0 & 0 & a_{12}a_{24} & 0 & 0 & 0 & 0  \\
0 & 0 & \alpha & 0 & 0 & 0  & 0 \\
0 & 0 & 0 & a_{56}\alpha+a_{12}a_{24}a_{46} & 0 & 0 & 0 \\  
0 & 0 & 0 & a_{57}\alpha & 0 & 0 & 0
\end{bmatrix},
$}
\end{equation}
where $\alpha=a_{12}a_{25}+a_{13}a_{35}$.
Both $[\mathcal{C}]_{6,4}$ and $[\mathcal{C}]_{7,4}$ are \textit{multi-terms} in the same column.
However, due to the dependency on network parameters, only one of these can become zero.
For instance, setting $\alpha$ to zero makes $[\mathcal{C}]_{7,4}$ zero, but $[\mathcal{C}]_{6,4}$ becomes a \textit{single-term}, which is invariably non-zero.
Note that in this example, node 5 corresponds to $i_w$ in \textit{Condition~\ref{unique_stem_condition}}, and its \textit{multi-term} is $\alpha$.}
Hence, we focus on columns that contain only \textit{multi-terms}, including the \textit{multi-term} of $i_w$.
The following lemma provides the necessary and sufficient conditions for such a column to become a zero vector.

\begin{lemma}\label{lem_mul}
For a graph $\mathcal{G}(\mathcal{V},\mathcal{E})$ with an integrator $i_w\in\mathcal{V}^{int}$ satisfying \textit{Condition~\ref{unique_stem_condition}},
consider a column $c_k$ in the controllability matrix that consists only of \textit{multi-terms} including the \textit{multi-term} of $i_w$.
All the \textit{multi-terms} in column $c_k$ can become zero if and only if
there is no intermediator among the nodes corresponding to these \textit{multi-terms} that passes through the integrator $i_l$ from the leader.
\end{lemma}

\begin{proof}
Fig.~\ref{fig_cond}(b) supports the understanding of this proof.
For \textit{only if} condition, in $c_k$ containing the \textit{multi-term} of $i_w$, 
suppose that there exists a \textit{multi-term} of an intermediator $j_l \in \mathcal{V}^{med}$ passing through $i_l$ from the leader.
Then, the \textit{multi-term} of $j_l$ is determined by the product of the \textit{multi-term} of $i_l$ and the WP of the path from $i_l$ to $j_l$. 
This prevents the \textit{multi-terms} of both $i_w$ and $j_l$ from being simultaneously zero, thereby ensuring $c_k$ cannot become a zero vector.
For \textit{if} condition, suppose that there is no \textit{multi-term} within $c_k$ corresponding to any intermediator passing through the integrator $i_l$ from the leader.
Since there exists a stem leading to $i_w$ that includes $i_l$, 
the column $c_k$ cannot include both the \textit{multi-term} of $i_w$ and any \textit{multi-term} that is derived from the \textit{multi-term} of $i_l$ by multiplication with a WP. 
Therefore, from \textit{Condition~\ref{unique_stem_condition}}, \textit{multi-terms} in $c_k$ can always be zero.
\end{proof}

\textcolor{black}{
With the above lemmas, we can conclude that 
for a specific column of the controllability matrix to become a zero vector, any non-zero elements must be \textit{multi-terms}.}
Moreover, the \textit{multi-terms} in this column must not include any intermediator passing through the integrator $i_l$.
This indicates that the presence or absence of such intermediators is crucial in determining whether the column can become a zero vector.

\begin{remark}\label{rem_mul}
Within the controllability matrix,
consider a column $c_k$ for $k\in\{1,...,n\}$ that satisfies \textit{Lemma~\ref{lem_mul}}.
Then, every node reachable from the leader in $(k-1)$-steps, is either an integrator or an intermediator, and the \textit{multi-terms} associated with these nodes can be zero.
Since nodes reachable from the leader in more than $(k-1)$-steps must pass through one of these integrators or intermediators, their \textit{multi-terms} are also trivially zeroable.
Consequently, if a specific column $c_k$ becomes a zero vector, then all subsequent columns in the controllability matrix can become zero vectors as well.
\end{remark}

\subsection{Method for Modified Controllability Matrix}\textcolor{black}{
For the controllability matrix $\mathcal{C}$ of a graph $\mathcal{G}(\mathcal{V},\mathcal{E})$, we propose a method to construct a modified controllability matrix $\bar{\mathcal{C}}$. 
Drawing insights from \textit{Lemma~\ref{lem_del}} and \textit{Lemma~\ref{lem_mul}}, 
the objective is to construct $\bar{\mathcal{C}}$ by specifically zeroing out the \textit{multi-terms} in $\mathcal{C}$ that potentially decrease its maximum rank.
The method begins with identifying the first column in $\mathcal{C}$ satisfying \textit{Lemma~\ref{lem_mul}}, which can become a zero vector. 
Once this column is identified, as discussed in \textit{Remark~\ref{rem_mul}}, 
the next step involves zeroing out the identified column and all subsequent columns. 
This targeted removal effectively simplifies the matrix by eliminating elements that contribute to its maximum rank, thereby maintaining the minimum rank characteristics of $\mathcal{C}$.
By applying this procedure, the modified controllability matrix $\bar{\mathcal{C}}$ is obtained, which retains the essential properties needed for further analysis while discarding unnecessary complexities.}
The following theorem provides the relationship between the maximum rank of $\bar{\mathcal{C}}$ and the minimum rank of $\mathcal{C}$:

\begin{theorem}\label{thm_multi}
The minimum rank of $\mathcal{C}$ and the maximum rank of the modified controllability matrix $\bar{\mathcal{C}}$ are identical.
\end{theorem}
\begin{proof}
Under \textit{Assumption~\ref{assum}}, each row of $\mathcal{C}$ is either a zero vector or contains exactly one non-zero element. 
From \textit{Lemma~\ref{lem_del}}, we know that the maximum rank of $\mathcal{C}$ decreases only when a column composed entirely of \textit{multi-terms} becomes zero.
Given that columns $c_k$ for $k \in \{1, \ldots, n\}$ containing only \textit{single-terms} do not decrease the maximum rank, we can categorize the columns into the following two cases:

\begin{itemize}
\item \textbf{Case 1:} $c_k$ contains both \textit{single-term} and \textit{multi-term}.
\item \textbf{Case 2:} $c_k$ contains only \textit{multi-term}.
\end{itemize}
In \textbf{Case 1}, even assuming \textit{multi-terms} to be zero does not decrease the maximum rank due to the remaining \textit{single-term}, as outlined in \textit{Lemma~\ref{lem_del}}. 
For \textbf{Case 2}, suppose that $c_k$ does not satisfy \textit{Lemma~\ref{lem_mul}}. 
It follows that there exists a \textit{multi-term} in $c_k$ that cannot be zero, which means that $c_k$ cannot become a zero vector, and therefore does not decrease the maximum rank.
Conversely, suppose that $c_k$ satisfies \textit{Lemma~\ref{lem_mul}}.
Then, all \textit{multi-terms} within a column can be zero, transforming such a column into a zero vector, thereby always reducing the maximum rank.
Therefore, the modified controllability matrix $\bar{\mathcal{C}}$ by zeroing only the \textit{multi-terms} that contribute to the decrease of the maximum rank, 
ensures that its maximum rank reflects the minimum rank of the original matrix $\mathcal{C}$.
\end{proof}

Note that since $\bar{\mathcal{C}}$ is constructed by eliminating \textit{multi-terms} in $\mathcal{C}$ that contribute to the decrease of the maximum rank, 
the minimum and maximum ranks of $\bar{\mathcal{C}}$ are identical.

\subsection{Subgraph Construction for the Dimension of SSCS}
The process of constructing the modified controllability matrix offers insights from both a matrix perspective and a graph-theoretical viewpoint. 
Leveraging this dual perspective, we can construct a subgraph $\bar{\mathcal{G}}(\mathcal{V},\bar{\mathcal{E}})$ that parallels the construction of the modifed controllability matrix, 
as detailed in \textit{Algorithm~\ref{alg_sub}}.

\begin{algorithm}[h]
\caption{\textcolor{black}{Constructing Subgraph Corresponding to $\bar{C}$}}\label{alg_sub}
\begin{algorithmic}[1]
\State \textbf{Input:} A graph $\mathcal{G}(\mathcal{V},\mathcal{E})$ with $|\mathcal{V}|=n$
\State Initialize $\mathcal{E}'= \emptyset$ 

\For{$k = 1$ \textbf{to} $n-1$}  \Comment{\textbf{Stage 1}}
    \State $\mathcal{V}_k \leftarrow$ nodes reachable from the leader with $k$-steps
    \If{$\mathcal{V}_k$ consists only of integrators or intermediators, \newline \hspace*{0.77cm} with no intermediators satisfying \textit{Lemma~\ref{lem_mul}}}
        \State $k_{first} \leftarrow k$ and \textbf{break} 
    \EndIf
\EndFor 

\For{$k = k_{first}$ \textbf{to} $n-1$} \Comment{\textbf{Stage 2}}
    \State $\mathcal{V}_k \leftarrow$ nodes reachable from the leader with $k$-steps
    \State $\mathcal{E}_k \leftarrow$ edges connected with nodes in $\mathcal{V}_k$
    \State $\mathcal{E}'\leftarrow \mathcal{E}' \cup \mathcal{E}_k$ 
\EndFor 
\State $\bar{\mathcal{E}} \leftarrow \mathcal{E}\setminus\mathcal{E}'$
\State \textbf{Output:} Subgraph $\bar{\mathcal{G}}(\mathcal{V},\bar{\mathcal{E}})$
\end{algorithmic}
\end{algorithm}




\textcolor{black}{
\textit{Algorithm~\ref{alg_sub}} mirrors the stages outlined in the process of constructing the modified controllability matrix, translating matrix modifications into graph-based modifications. 
The initial stage involves identifying the first step from the leader where only integrators or intermediators satisfying \textit{Lemma~\ref{lem_mul}} are reachable.
Following the identification of the first step, and based on \textit{Remark~\ref{rem_mul}}, 
the second stage concentrates on all nodes that become reachable beyond this step.
This leads to the removal of edges connected to these nodes, resulting in the construction of $\bar{\mathcal{G}}(\mathcal{V},\bar{\mathcal{E}})$, 
a subgraph that reflects the modifications applied to the controllability matrix. 
This ensures that the structural characteristics of the subgraph are aligned with those of the modified controllability matrix $\bar{\mathcal{C}}$.
}


Drawing from \textit{Theorem~\ref{thm_multi}}, this integration of graph theory and matrix analysis leads us to the following corollary.

\begin{corollary}\label{cor_SSCS}
The dimension of SSCS of a graph $\mathcal{G}(\mathcal{V},\mathcal{E})$ is identical to 
the dimension of SCS of the subgraph $\bar{\mathcal{G}}(\mathcal{V},\bar{\mathcal{E}})$ obtained from \textit{Algorithm~\ref{alg_sub}}.
\end{corollary}

The above corollary implies that the dimension of SSCS for a graph $\mathcal{G}(\mathcal{V},\mathcal{E})$ 
can be determined by applying \textit{Proposition~\ref{thm_hosoe}} to its subgraph $\bar{\mathcal{G}}(\mathcal{V},\bar{\mathcal{E}})$.
This foundation allows us to establish the following proposition to determine the dimension of SSCS:

\begin{proposition}\label{thm_minimum}\textcolor{black}{
For a graph $\mathcal{G}(\mathcal{V},\mathcal{E})$,
the dimension of SSCS is the maximum number of state nodes that can be covered by a disjoint set of stems in the subgraph $\bar{\mathcal{G}}({\mathcal{V}},\bar{\mathcal{E}})$.
}
\end{proposition}

\textcolor{black}{
From \textit{Theorem~\ref{thm_FSSC}} and \textit{Proposition~\ref{thm_minimum}}, we now have a clear pathway to find FSSC nodes in a graph.
For an HDAG, the determination of FSSC nodes by \textit{Theorem~\ref{thm_FSSC}} 
involves analyzing how the dimension of SSCS changes when a state node becomes a leader with additional input.}
For a detailed comparison between the zero forcing set-based lower bounds of the dimension of SSCS in \cite{monshizadeh2015strong,yaziciouglu2020strong,yazicioglu2012tight,yaziciouglu2016graph} 
and the results of \textit{Proposition~\ref{thm_minimum}}, as well as its complexity analysis, please see \textit{Sections~IV} and \textit{V.B} in the Supplementary Materials.

\section{Examples and Applications}\label{sec_ex}
\textcolor{black}{
This section presents a topological example to illustrate the determination of FSC and FSSC nodes and discusses their applications. 
Consider a graph $\mathcal{G}(\mathcal{V},\mathcal{E})$ and its corresponding subgraph $\bar{\mathcal{G}}(\mathcal{V},\bar{\mathcal{E}})$ as shown in Fig.~\ref{fig3}.
Let the sets of FSC and FSSC nodes be denoted by $\mathcal{V}^{\text{\tiny FSC}}$ and $\mathcal{V}^{\text{\tiny FSSC}}$, respectively.
From \textit{Proposition~\ref{thm_commault}}, we identify the FSC nodes as $\mathcal{V}^{\text{\tiny FSC}}=\{1,2,13\}$, 
where the dimension of SCS does not increase when these nodes become leaders.
To determine FSSC nodes, we analyze the subgraph $\bar{\mathcal{G}}(\mathcal{V},\bar{\mathcal{E}})$, constructed by removing edges connected to nodes reachable from the leader in or after $4$-steps,
as illustrated in Fig.\ref{fig3}(b).
Applying \textit{Proposition~\ref{thm_minimum}} on this subgraph gives the SSCS dimension for the original graph $\mathcal{G}(\mathcal{V},\mathcal{E})$.
Hence, by \textit{Theorem~\ref{thm_FSSC}}, the FSSC nodes are those that do not increase the SSCS dimension when they become leaders, resulting in $\mathcal{V}^{\text{\tiny FSSC}}=\{1,2\}$.
}

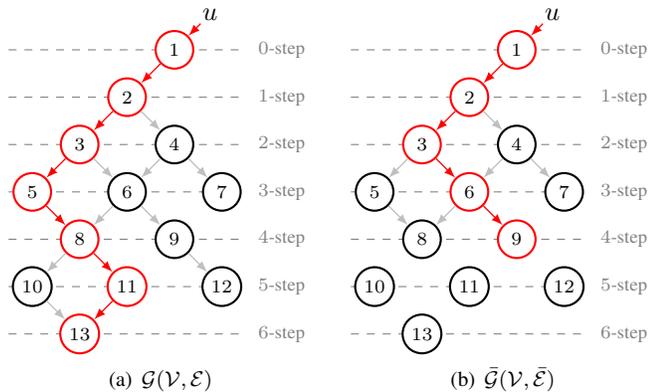
\begin{figure}[t]
\centering
\subfigure[${\mathcal{G}}(\mathcal{V},{\mathcal{E}})$]{
\begin{tikzpicture}[scale=0.63]
\node[] at (2.25,5.75) {$u$};
\node[] (node0) at (2.25,5.75) [] {};

\draw[dashed, gray] (-2,5) -- (3,5) node[left, pos=1.29] { \scriptsize$0$-step};
\draw[dashed, gray] (-2,4) -- (3,4) node[left, pos=1.29] { \scriptsize$1$-step};
\draw[dashed, gray] (-2,3) -- (3,3) node[left, pos=1.29] { \scriptsize$2$-step};
\draw[dashed, gray] (-2,2) -- (3,2) node[left, pos=1.29] { \scriptsize$3$-step};
\draw[dashed, gray] (-2,1) -- (3,1) node[left, pos=1.29] { \scriptsize$4$-step};
\draw[dashed, gray] (-2,0) -- (3,0) node[left, pos=1.29] { \scriptsize$5$-step};
\draw[dashed, gray] (-2,-1) -- (3,-1) node[left, pos=1.29] { \scriptsize$6$-step};

\node[place, draw=red, circle,minimum size=0.5cm] (node1) at (1.5,5) [] {\scriptsize$1$}; 
\node[place, draw=red, circle,minimum size=0.5cm] (node2) at (0.5,4) [] {\scriptsize$2$};
\node[place, draw=red, circle,minimum size=0.5cm] (node3) at (-0.5,3) [] {\scriptsize$3$};
\node[place, circle,minimum size=0.5cm] (node4) at (1.5,3) [] {\scriptsize$4$};
\node[place, draw=red, circle,minimum size=0.5cm] (node5) at (-1.5,2) [] {\scriptsize$5$};
\node[place, circle,minimum size=0.5cm] (node6) at (0.5,2) [] {\scriptsize$6$};
\node[place, circle,minimum size=0.5cm] (node7) at (2.5,2) [] {\scriptsize$7$};
\node[place, draw=red, circle,minimum size=0.5cm] (node8) at (-0.5,1) [] {\scriptsize$8$};
\node[place, circle,minimum size=0.5cm] (node9) at (1.5,1) [] {\scriptsize$9$};
\node[place, circle,minimum size=0.5cm] (node10) at (-1.5,0) [] {\scriptsize$10$};
\node[place, draw=red, circle,minimum size=0.5cm] (node11) at (0.5,0) [] {\scriptsize$11$};
\node[place, circle,minimum size=0.5cm] (node12) at (2.5,0) [] {\scriptsize$12$};
\node[place, draw=red, circle,minimum size=0.5cm] (node13) at (-0.5,-1) [] {\scriptsize$13$};

\draw (node0) [-latex, red, line width=0.5pt] -- node [right]  {} (node1);
\draw (node1) [-latex, red, line width=0.5pt] -- node [right]  {} (node2);
\draw (node2) [-latex, red, line width=0.5pt] -- node [right]  {} (node3);
\draw (node2) [-latex, lightgray, line width=0.5pt] -- node [right]  {} (node4);
\draw (node3) [-latex, red, line width=0.5pt] -- node [right]  {} (node5);
\draw (node3) [-latex, lightgray, line width=0.5pt] -- node [right]  {} (node6);
\draw (node4) [-latex, lightgray, line width=0.5pt] -- node [right]  {} (node6);
\draw (node4) [-latex, lightgray, line width=0.5pt] -- node [right]  {} (node7);
\draw (node5) [-latex, red, line width=0.5pt] -- node [right]  {} (node8);
\draw (node6) [-latex, lightgray, line width=0.5pt] -- node [right]  {} (node8);
\draw (node6) [-latex, lightgray, line width=0.5pt] -- node [right]  {} (node9);
\draw (node8) [-latex, lightgray, line width=0.5pt] -- node [right]  {} (node10);
\draw (node8) [-latex, red, line width=0.5pt] -- node [right]  {} (node11);
\draw (node9) [-latex, lightgray, line width=0.5pt] -- node [right]  {} (node12);
\draw (node10) [-latex, lightgray, line width=0.5pt] -- node [right]  {} (node13);
\draw (node11) [-latex, red, line width=0.5pt] -- node [right]  {} (node13);
\end{tikzpicture}
}\,\,
\subfigure[$\bar{\mathcal{G}}(\mathcal{V},\bar{\mathcal{E}})$]{
\begin{tikzpicture}[scale=0.63]
\node[] at (2.25,5.75) {$u$};
\node[] (node0) at (2.25,5.75) [] {};

\draw[dashed, gray] (-2,5) -- (3,5) node[left, pos=1.29] { \scriptsize$0$-step};
\draw[dashed, gray] (-2,4) -- (3,4) node[left, pos=1.29] { \scriptsize$1$-step};
\draw[dashed, gray] (-2,3) -- (3,3) node[left, pos=1.29] { \scriptsize$2$-step};
\draw[dashed, gray] (-2,2) -- (3,2) node[left, pos=1.29] { \scriptsize$3$-step};
\draw[dashed, gray] (-2,1) -- (3,1) node[left, pos=1.29] { \scriptsize$4$-step};
\draw[dashed, gray] (-2,0) -- (3,0) node[left, pos=1.29] { \scriptsize$5$-step};
\draw[dashed, gray] (-2,-1) -- (3,-1) node[left, pos=1.29] { \scriptsize$6$-step};

\node[place, draw=red, circle,minimum size=0.5cm] (node1) at (1.5,5) [] {\scriptsize$1$}; 
\node[place, draw=red, circle,minimum size=0.5cm] (node2) at (0.5,4) [] {\scriptsize$2$};
\node[place, draw=red, circle,minimum size=0.5cm] (node3) at (-0.5,3) [] {\scriptsize$3$};
\node[place, circle,minimum size=0.5cm] (node4) at (1.5,3) [] {\scriptsize$4$};
\node[place, circle,minimum size=0.5cm] (node5) at (-1.5,2) [] {\scriptsize$5$};
\node[place, draw=red, circle,minimum size=0.5cm] (node6) at (0.5,2) [] {\scriptsize$6$};
\node[place, circle,minimum size=0.5cm] (node7) at (2.5,2) [] {\scriptsize$7$};
\node[place, circle,minimum size=0.5cm] (node8) at (-0.5,1) [] {\scriptsize$8$};
\node[place, draw=red, circle,minimum size=0.5cm] (node9) at (1.5,1) [] {\scriptsize$9$};
\node[place, circle,minimum size=0.5cm] (node10) at (-1.5,0) [] {\scriptsize$10$};
\node[place, circle,minimum size=0.5cm] (node11) at (0.5,0) [] {\scriptsize$11$};
\node[place, circle,minimum size=0.5cm] (node12) at (2.5,0) [] {\scriptsize$12$};
\node[place, circle,minimum size=0.5cm] (node13) at (-0.5,-1) [] {\scriptsize$13$};

\draw (node0) [-latex, red, line width=0.5pt] -- node [right]  {} (node1);
\draw (node1) [-latex, red, line width=0.5pt] -- node [right]  {} (node2);
\draw (node2) [-latex, red, line width=0.5pt] -- node [right]  {} (node3);
\draw (node2) [-latex, lightgray, line width=0.5pt] -- node [right]  {} (node4);
\draw (node3) [-latex, lightgray, line width=0.5pt] -- node [right]  {} (node5);
\draw (node3) [-latex, red, line width=0.5pt] -- node [right]  {} (node6);
\draw (node4) [-latex, lightgray, line width=0.5pt] -- node [right]  {} (node6);
\draw (node4) [-latex, lightgray, line width=0.5pt] -- node [right]  {} (node7);
\draw (node5) [-latex, lightgray, line width=0.5pt] -- node [right]  {} (node8);
\draw (node6) [-latex, lightgray, line width=0.5pt] -- node [right]  {} (node8);
\draw (node6) [-latex, red, line width=0.5pt] -- node [right]  {} (node9);
\end{tikzpicture}
}
\caption{
\textcolor{black}{
A graph $\mathcal{G}(\mathcal{V},\mathcal{E})$ and its corresponding subgraph $\bar{\mathcal{G}}(\mathcal{V},\bar{\mathcal{E}})$ obtained from \textit{Algorithm~\ref{alg_sub}}:
(a) From \textit{Proposition~\ref{thm_hosoe}}, the dimension of SCS for $\mathcal{G}(\mathcal{V},\mathcal{E})$ is 7.
(b) From \textit{Proposition~\ref{thm_minimum}}, the dimension of SSCS for $\mathcal{G}(\mathcal{V},\mathcal{E})$ is 5.
}}
\label{fig3}
\end{figure}

\textcolor{black}{
The concepts of FSC and FSSC nodes offer a framework for assessing node controllability and importance, especially in the face of network parameter variations.
FSC nodes provide baseline controllability, but FSSC nodes maintain consistent controllability across all network parameters.
This distinction is critical in task or role assignment based on node significance.
For instance, in a network of multiple agents, prioritizing FSSC nodes for crucial tasks enhances operational efficiency and system resilience against parameter fluctuations.
Hence, FSC and FSSC nodes can serve as metrics for establishing a hierarchy of node importance, optimizing task distribution, and ensuring a robust control system resilient to diverse parameter changes.
}

\section{Conclusion} \label{sec_conc}
This paper introduced the concept of FSSCS and its associated FSSC nodes. 
Through a combination of graph-theoretical and controllability matrix analyses, we have provided a comprehensive understanding of the strong structural controllability of individual nodes. 
This approach not only extends the existing framework of FSC nodes
but also offers new insights into the robustness of network controllability against variations in network parameters. 
Our findings have enabled the complete characterization of the controllable subspace within structured networks and established a solid foundation for understanding the controllability of individual nodes. 
While the exact determination of the SSCS dimension is currently limited to HDAGs, we conjecture that similar principles may hold for general directed acyclic graphs. 
\textcolor{black}{
Additionally, achieving the full characterization of FSSCS, including proof of how it is generated in general graphs, remains an open challenge. 
This requires developing methods to determine the dimension of SSCS for arbitrary graph structures, which will be a critical focus of our future work.
}

\bibliographystyle{unsrt}
\bibliography{references_manuscript}

\end{document}